\documentclass[12pt,a4paper]{amsart}

\usepackage[hmargin=2.0cm,vmargin=2cm]{geometry}

\usepackage{amsmath,amsfonts,amssymb}

\usepackage[driverfallback=hypertex]{hyperref}
\usepackage{nameref,zref-xr}                    

\newcommand{\Z}{\mathbb Z}
\newcommand{\mbC}{\mathbb C}

\newcommand{\mcM}{\mathcal M}
\newcommand{\oM}{\overline{\mathcal M}}

\newcommand{\<}{\left <}
\renewcommand{\>}{\right >}

\def\d{\partial}
\def\CP1{\mathbb{C}\mathrm{P}^1}

\newcommand{\cL}{\mathcal L}

\newcommand{\mcP}{\mathcal P}
\newcommand{\mcA}{\mathcal A}
\newcommand{\mcE}{\mathcal E}
\newcommand{\tR}{\widetilde R}
\newcommand{\tQ}{\widetilde Q}

\newtheorem{theorem}{Theorem}[section]
\newtheorem{proposition}[theorem]{Proposition}
\newtheorem{lemma}[theorem]{Lemma}

\theoremstyle{definition}

\title[Equivalence of the open KdV and the open Virasoro]{Equivalence of the open KdV and the open Virasoro equations for the moduli space of Riemann surfaces with boundary}

\author{Alexandr Buryak}

\address{Alexandr~Buryak:\newline
Department of Mathematics,
ETH Zurich,\newline
Ramistrasse 101 8092, HG G27.1, Zurich, Switzerland.} 
\email{buryaksh@gmail.com}



\numberwithin{equation}{section}

\begin{document}

\begin{abstract}

In a recent paper R. Pandharipande, J. Solomon and R. Tessler initiated a study of the intersection theory on the moduli space of Riemann surfaces with boundary. The authors conjectured KdV and Virasoro type equations that completely determine all intersection numbers. In this paper we study these equations in detail. In particular, we prove that the KdV and the Virasoro type equations for the intersection numbers on the moduli space of Riemann surfaces with boundary are equivalent.   

\end{abstract}

\maketitle

\section{Introduction}

Denote by~$\mcM_{g,n}$ the moduli space of smooth complex algebraic curves of genus~$g$ with~$n$ distinct marked points. In~\cite{DM69} P.~Deligne and D.~Mumford defined a natural compactification $\mcM_{g,n}\subset\oM_{g,n}$ via stable curves (with possible nodal singularities). The moduli space $\oM_{g,n}$ is a nonsingular complex orbifold of dimension $3g-3+n$.

A new direction in the study of the moduli space $\oM_{g,n}$ was opened by E. Witten~\cite{Wit91}. The class $\psi_i\in H^2(\oM_{g,n};\mbC)$ is defined as the first Chern class of the line bundle over $\oM_{g,n}$ formed by the cotangent lines at the $i$-th marked point. Intersection numbers $\<\tau_{k_1}\tau_{k_2}\ldots\tau_{k_n}\>^c_g$ are defined as follows:
$$
\<\tau_{k_1}\tau_{k_2}\ldots\tau_{k_n}\>^c_g:=\int_{\oM_{g,n}}\psi_1^{k_1}\psi_2^{k_2}\ldots\psi_n^{k_n}.
$$ 
The superscript $c$ here signals integration over the moduli of closed Riemann surfaces. Let us introduce variables $u,t_0,t_1,t_2,\ldots$ and consider the generating series 
$$
F^c(t_0,t_1,\ldots;u):=\sum_{\substack{g\ge 0,n\ge 1\\2g-2+n>0}}\frac{u^{2g-2}}{n!}\sum_{k_1,\ldots,k_n\ge 0}\<\tau_{k_1}\tau_{k_2}\ldots\tau_{k_n}\>^c_gt_{k_1}t_{k_2}\ldots t_{k_n}.
$$
E.~Witten (\cite{Wit91}) proved that the generating series~$F^c$ satisfies the so-called string equation and conjectured that the second derivative $\frac{\d^2 F^c}{\d t_0^2}$ is a solution of the KdV hierarchy. Witten's conjecture was proved by M.~Kontsevich (\cite{Kon92}). 

There is a reformulation of Witten's conjecture due to R.~Dijkgraaf, E.~Verlinde and H.~Verlinde (\cite{DVV91}) in terms of the Virasoro algebra. They defined certain quadratic differential operators $L_n$, $n\ge -1$, and proved that Witten's conjecture is equivalent to the equations 
\begin{gather}\label{eq: virasoro}
L_n\exp(F^c)=0,
\end{gather}
that are called the Virasoro equations. The operators $L_n$ satisfy the relation $[L_n,L_m]=(n-m)L_{n+m}$.  

In \cite{PST14} the authors initiated a study of the intersection theory on the moduli space of Riemann surfaces with boundary. They introduced intersection numbers on this moduli space and completely described them in genus~$0$. In higher genera the authors conjectured that the generating series of the intersection numbers satisfies certain partial differential equations that are analagous to the string, the KdV and the Virasoro equations. In~\cite{PST14} these equations were called the open string, the open KdV and the open Virasoro equations. 

The open KdV equations and the open Virasoro equations provide two different ways to describe the intersection numbers on the moduli space of Riemann surfaces with boundary. It is absolutely non-obvious that these two descriptions are equivalent and it was left in~\cite{PST14} as a conjecture. The main purpose of this paper is to prove this conjecture. We show that the system of the open KdV equations has a unique solution specified by a certain initial condition that corresponds to the simplest intersection numbers in genus $0$. The main result of the paper is the proof of the fact that this solution satisfies the open Virasoro equations. This proves that the open KdV and the open Virasoro equations give equivalent descriptions of the intersection numbers on the moduli space of Riemann surfaces with boundary. 


\subsection{Witten's conjecture and the Virasoro equations}\label{subsection:Witten conjecture}

In this section we review Witten's conjecture and its reformulation due to R.~Dijkgraaf, E.~Verlinde and H.~Verlinde. 

One of the basic properties of the generating series~$F^c$ is the so-called string equation~(\cite{Wit91}):
\begin{gather}\label{eq: string equation}
\frac{\d F^c}{\d t_0}=\sum_{n\ge 0}t_{n+1}\frac{\d F^c}{\d t_n}+\frac{t_0^2}{2u^2}.
\end{gather}

\subsubsection{KdV equations}

E.~Witten conjectured~(\cite{Wit91}) that the generating series $F^c$ is the logarithm of a tau-function of the KdV hierarchy. In particular, it means that it satisfies the following system:
\begin{gather}\label{eq:tau function}
u^{-2}\frac{2n+1}{2}\frac{\d^3 F^c}{\d t_0^2\d t_n}=\frac{\d^2 F^c}{\d t_0^2}\frac{\d^3 F^c}{\d t_0^2\d t_{n-1}}+\frac{1}{2}\frac{\d^3 F^c}{\d t_0^3}\frac{\d^2 F^c}{\d t_0\d t_{n-1}}+\frac{1}{8}\frac{\d^5 F^c}{\d t_0^4\d t_{n-1}},\quad n\ge 1.
\end{gather}
Moreover, E.~Witten showed~(\cite{Wit91}) that the KdV equations~\eqref{eq:tau function} together with the string equation~\eqref{eq: string equation} and the initial condition $\left.F^c\right|_{t_*=0}=0$ uniquely determine the power series~$F^c$.  

\subsubsection{Virasoro equations}

The Virasoro operators $L_n,n\ge -1$, are defined as follows:
\begin{multline*}
L_n:=\sum_{i\ge 0}\frac{(2i+2n+1)!!}{2^{n+1}(2i-1)!!}(t_i-\delta_{i,1})\frac{\d}{\d t_{i+n}}+\frac{u^2}{2}\sum_{i=0}^{n-1}\frac{(2i+1)!!(2n-2i-1)!!}{2^{n+1}}\frac{\d^2}{\d t_i\d t_{n-1-i}}\\
+\delta_{n,-1}\frac{t_0^2}{2 u^2}+\delta_{n,0}\frac{1}{16}.
\end{multline*}
They satisfy the commutation relation
\begin{gather}\label{eq: commutation}
[L_n,L_m]=(n-m)L_{n+m}.
\end{gather}
The Virasoro equations say that 
\begin{gather}\label{eq: Virasoro}
L_n\exp(F^c)=0,\quad n\ge -1.
\end{gather}
For $n=-1$, this equation is equivalent to the string equation~\eqref{eq: string equation}.

R.~Dijkgraaf, E.~Verlinde and H.~Verlinde (\cite{DVV91}) proved that Witten's conjecture is equivalent to the Virasoro equations. To be precise, they proved the following. Suppose a power series $F$ satisfies the string equation~\eqref{eq: string equation} and the KdV equations~\eqref{eq:tau function}. Then $F$ satisfies the Virasoro equations~\eqref{eq: Virasoro}. We review the proof of this fact in Appendix~\ref{section: closed virasoro}.


\subsection{Moduli of Riemann surfaces with boundary}

Here we briefly recall the basic definitions concerning the moduli space of Riemann surfaces with boundary. We refer the reader to~\cite{PST14} for details. 

Let $\Delta\in\mbC$ be the open unit disk, and let $\overline\Delta$ be its closure. An extendable embedding of the open disk~$\Delta$ in a closed Riemann surface $f\colon\Delta\to C$ is a holomorphic map which extends to a holomorphic embedding of an open neighbourhood of $\overline\Delta$. Two extendable embeddings are disjoint, if the images of $\overline\Delta$ are disjoint.

A Riemann surface with boundary $(X,\d X)$ is obtained by removing a finite positive number of disjoint extendable open disks from a connected compact Riemann surface. A compact Riemann surface is not viewed here as Riemann surface with boundary.

Given a Riemann surface with boundary $(X,\d X)$, we can canonically construct the double via the Schwartz reflection through the boundary. The double $D(X,\d X)$ of $(X,\d X)$ is a closed Riemann surface. The doubled genus of $(X,\d X)$ is defined to be the usual genus of $D(X,\d X)$.

On a Riemann surface with boundary $(X,\d X)$, we consider two types of marked points. The markings of interior type are points of $X\backslash\d X$. The markings of boundary type are points of~$\d X$. Let~$\mcM_{g,k,l}$ denote the moduli space of Riemann surfaces with boundary of doubled genus~$g$ with~$k$ distinct boundary markings and $l$ distinct interior markings. The moduli space~$\mcM_{g,k,l}$ is defined to be empty unless the stability condition $2g-2+k+2l>0$ is satisfied. The moduli space $\mcM_{g,k,l}$ is a real orbifold of real dimension $3g-3+k+2l$. 

The psi-classes $\psi_i\in H^2(\mcM_{g,k,l};\mbC)$ are defined as the first Chern classes of the cotangent line bundles for the interior markings. The authors of \cite{PST14} do not consider the cotangent lines at boundary points. Naively, open intersection numbers are defined by
\begin{gather}\label{eq: open intersections}
\left<\tau_{a_1}\tau_{a_2}\ldots\tau_{a_l}\sigma^k\right>^o_g:=\int_{\oM_{g,k,l}}\psi_1^{a_1}\psi_2^{a_2}\ldots\psi_l^{a_l}.
\end{gather}
To rigorously define the right-hand side of~\eqref{eq: open intersections}, at least three significant steps must be taken:
\begin{itemize}

\item A natural compactification $\mcM_{g,k,l}\subset\oM_{g,k,l}$ must be constructed. Candidates for $\oM_{g,k,l}$ are themselves real orbifolds with boundary $\d\oM_{g,k,l}$.

\item For integration over $\oM_{g,k,l}$ to be well-defined, boundary conditions of the integrand along $\d\oM_{g,k,l}$ must be specified.

\item Problems with an orientation should be solved, since the moduli space $\mcM_{g,k,l}$ is in general non-orientable.  

\end{itemize}
The authors of~\cite{PST14} completed all these steps and rigorously defined open intersection numbers in genus~$0$. Moreover, they obtained a complete description of them. In higher genera, even though open intersection numbers are not well-defined, the authors of~\cite{PST14} proposed a beautiful conjectural description of them that we are going to recall in the next section. 


\subsection{Open KdV and open Virasoro equations}

In this section we review the KdV and the Virasoro type equations from \cite{PST14} for the open intersection numbers~\eqref{eq: open intersections}. 

Introduce one more formal variable $s$ and define the generating series $F^o$ by
$$
F^o(t_0,t_1,\ldots,s;u):=\sum_{\substack{g,k,l\ge 0\\2g-2+k+2l>0}}\frac{u^{g-1}}{k!l!}\sum_{a_1,\ldots,a_l\ge 0}\<\tau_{a_1}\ldots\tau_{a_l}\sigma^k\>^o_gt_{a_1}\ldots t_{a_n} s^k.
$$ 
First of all, the authors of~\cite{PST14} conjectured the following analog of the string equation~\eqref{eq: string equation}:
\begin{gather}\label{eq: open string}
\frac{\d F^o}{\d t_0}=\sum_{i\ge 0}t_{i+1}\frac{\d F^o}{\d t_i}+u^{-1}s.
\end{gather}
They call it the open string equation. The authors also proved that the following initial condition holds:
\begin{gather}\label{eq: initial conditions}
\left.F^o\right|_{t_{\ge 1}=0}=u^{-1}\left(\frac{s^3}{6}+t_0 s\right).
\end{gather}

\subsubsection{Open KdV equations}

The authors of~\cite{PST14} conjectured that the generating series~$F^o$ satisfies the following system of equations: 
\begin{gather}\label{eq: open kdv}
\frac{2n+1}{2}\frac{\d F^o}{\d t_n}=u\frac{\d F^o}{\d s}\frac{\d F^o}{\d t_{n-1}}+u\frac{\d^2 F^o}{\d s\d t_{n-1}}+\frac{u^2}{2}\frac{\d F^o}{\d t_0}\frac{\d^2 F^c}{\d t_0\d t_{n-1}}-\frac{u^2}{4}\frac{\d^3 F^c}{\d t_0^2\d t_{n-1}},\quad n\ge 1.
\end{gather}
They call these equations the open KdV equations. It is clear that the open KdV equations~\eqref{eq: open kdv}, the initial condition~\eqref{eq: initial conditions} and the potential $F^c$ uniquely determine the series $F^o$. On the other hand, the existence of such a solution is completely non-obvious and we will prove it in this paper.

\subsubsection{Open Virasoro equations}

In~\cite{PST14} the authors introduced the following operators:
\begin{gather*}
\cL_n:=L_n+\left(u^n s\frac{\d^{n+1}}{\d s^{n+1}}+\frac{3n+3}{4}u^n\frac{\d^n}{\d s^n}\right),\quad n\ge -1.
\end{gather*}
These operators satisfy the same commutation relation as the operators $L_n$:
\begin{gather}\label{eq: open commutation}
[\cL_n,\cL_m]=(n-m)\cL_{n+m}.
\end{gather}

In \cite{PST14} the authors conjectured the following analog of the Virasoro equations~\eqref{eq: Virasoro}:
\begin{gather}\label{eq: open virasoro}
\cL_n\exp(F^o+F^c)=0,\quad n\ge -1.
\end{gather}
Clearly, equations~\eqref{eq: open virasoro}, the initial condition $F^o|_{t_*=0}=u^{-1}\frac{s^3}{6}$ and the potential $F^c$ completely determine the series $F^o$.


\subsection{Main result}

Here we formulate two main results of the paper.

\begin{theorem}\label{theorem: existence}
1. The system of the open KdV equations~\eqref{eq: open kdv} has a unique solution that satisfies the initial condition~\eqref{eq: initial conditions}.\\
2. This solution satisfies the following equation:
\begin{gather}\label{eq: s-string equation}
\frac{\d F^o}{\d s}=u\left(\frac{1}{2}\left(\frac{\d F^o}{\d t_0}\right)^2+\frac{1}{2}\frac{\d^2 F^o}{\d t_0^2}+\frac{\d^2 F^c}{\d t_0^2}\right).
\end{gather}

\end{theorem}

\begin{theorem}\label{theorem: main}
The series $F^o$ determined by Theorem \ref{theorem: existence} satisfies the open Virasoro equations~\eqref{eq: open virasoro}.
\end{theorem}


\subsection{Burgers-KdV hierarchy and descendants of $s$}

In Section~\ref{section:Burgers-KdV} we will construct a certain system of evolutionary partial differential equations with one spatial variable. It will be called the Burgers-KdV hierarchy. We will prove that the series $F^o$ determined by Theorem~\ref{theorem: existence} satisfies the half of the equations of this hierarchy. The remaining flows of the Burgers-KdV hierarchy suggest a way to introduce new variables $s_1,s_2,\ldots$ in the open potential $F^o$. These variables can be viewed as descendants of $s$. We hope that our idea can help to give a geometrical construction of descendants of $s$, at least in genus $0$.  


\subsection{Open KdV equations and the wave function of the KdV hierarchy}

In the work~\cite{Bur14}, that appeared while this paper was under consideration in the journal, we observed that the open KdV equations are closely related to the equations for the wave function of the KdV hierarchy. Using this observation our original proof of Theorem~\ref{theorem: existence} can be simplified. We discuss it in Section~\ref{subsection:wave}.

\subsection{Acknowledgements} 

The author would like to thank R. Pandharipande, J. Solomon and R. Tessler for very useful discussions. 

We would like to thank the anonymous referee for valuable remarks and suggestions that allowed us to improve the exposition of this paper.

The author was supported by grant ERC-2012-AdG-320368-MCSK in the group of R. Pandharipande at ETH Zurich, by Russian Federation Government grant no. 2010-220-01-077 (ag. no. 11.634.31.0005), the grants RFBR-10-010-00678, NSh-4850.2012.1, the Moebius Contest Foundation for Young Scientists and "Dynasty" foundation.


\subsection{Organization of the paper}

In Section~\ref{section:evolutionary pde} we recall some basic facts about evolutionary PDEs with one spatial variable and give a slight reformulation of Witten's conjecture. 

In Section~\ref{section:Burgers-KdV} we construct the Burgers-KdV hierarchy and prove that it has a solution for arbitrary polynomial initial conditions. We also construct a specific solution of the half of the Burgers-KdV hierarchy that satisfies the open string equation~\eqref{eq: open string}.

Section~\ref{section: existence} contains the proofs of Theorems \ref{theorem: existence} and \ref{theorem: burgers}. 

In Section~\ref{section: open virasoro} we prove Theorem~\ref{theorem: main}.

In Appendix~\ref{section: closed virasoro} we revisit the proof of the equivalence of the KdV and the Virasoro equations for the intersection numbers on the moduli space of stable curves. 


\section{Evolutionary partial differential equations}\label{section:evolutionary pde}

In this section we recall some basic facts about evolutionary PDEs with one spatial variable. We also review the construction of the KdV hierarchy and give a slight reformulation of Witten's conjecture that will be useful in the subsequent sections. All this material is well-known. We refer the reader to the book~\cite{Olv86} for the details about this subject.


\subsection{Ring of differential polynomials}

Let us fix an integer $N\ge 1$. Consider variables $v^i_j$, $1\le i\le N$, $j\ge 0$. We will often denote $v^i_0$ by $v^i$ and use an alternative notation for the variables $v^i_1,v^i_2,\ldots$:
$$
v^i_x:=v^i_1,\quad v^i_{xx}:=v^i_2,\ldots.
$$
Denote by $\mcA_{v^1,v^2,\ldots,v^N}$ the ring of polynomials in the variables $u,u^{-1}$ and $v^i_j$. The elements of~$\mcA_{v^1,\ldots,v^N}$ will be called differential polynomials. 

An operator $\d_x\colon\mcA_{v^1,\ldots,v^N}\to\mcA_{v^1,\ldots,v^N}$ is defined as follows:
$$
\d_x:=\sum_{i=1}^N\sum_{s\ge 0}v^i_{s+1}\frac{\d}{\d v^i_s}.
$$ 

Consider now a sequence of differential polynomials $P^i_j(v,v_x,\ldots;u)\in\mcA_{v^1,\ldots,v^N}$, $1\le i\le N$, $j\ge 0$. Consider the variables $v^i$ as formal power series in $x,\tau_0,\tau_1,\ldots$ with the coefficients from~$\mbC[u,u^{-1}]$. A system of evolutionary PDEs with one spatial variable is a system of the form:
\begin{gather}\label{eq: general system}
\frac{\d v^i}{\d\tau_j}=P^i_j(v,v_x,\ldots;u),\quad 1\le i\le N,\quad j\ge 0.
\end{gather}


\subsection{Existence of a solution}

Here we give a sufficient condition for system~\eqref{eq: general system} to have a solution. Let $P^1,\ldots,P^N\in\mcA_{v^1,\ldots,v^N}$ be some differential polynomials. Define an operator~$V_{P^1,\ldots,P^N}\colon\mcA_{v^1,\ldots,v^N}\to\mcA_{v^1,\ldots,v^N}$ by
\begin{gather*}
V_{P^1,\ldots,P^N}:=\sum_{i=1}^N\sum_{j\ge 0}(\d_x^j P^i)\frac{\d}{\d v^i_j}.
\end{gather*} 
Denote the space of all these operators by $\mcE_{v^1,\ldots,v^N}$. It is a Lie algebra: for $Q^1,Q^2,\ldots,Q^N\in\mcA_{v^1,\ldots,v^N}$, we have
\begin{gather*}
[V_{P^1,\ldots,P^N},V_{Q^1,\ldots,Q^N}]=V_{R^1,\ldots,R^N},\quad\text{where}\quad R^i=V_{P^1,\ldots,P^N}Q^i-V_{Q^1,\ldots,Q^N}P^i.
\end{gather*}

Consider an operator $O=\sum_{i\ge 0}O_i\d_x^i$, $O_i\in\mcA_{v^1,\ldots,v^N}$. We will use the following notation:
$$
V_{P_1,\ldots,P^N}\cdot O:=\sum_{i\ge 0}\left(V_{P^1,\ldots,P^N}O_i\right)\d_x^i.
$$ 

Let us again consider system~\eqref{eq: general system}. The following lemma is well-known (see e.g. \cite{Olv86}).
\begin{lemma}\label{lemma:existence}
Suppose that, for any $i,j\ge 0$, we have $[V_{P_i^1,\ldots,P_i^N},V_{P_j^1,\ldots,P_j^N}]=0$. Then, for an arbitrary initial condition $\left.v^i\right|_{\tau_j=0}=f^i(x,u)$, where $f^i(x,u)\in\mbC[x,u,u^{-1}]$, system~\eqref{eq: general system} has a unique solution.
\end{lemma}


\subsection{KdV hierarchy}\label{subsection:kdv hierarchy}

Consider a variable~$w$ and the ring $\mcA_w$. Define differential polynomials $K_n\in\mcA_{w}$, $n\ge 0$, by the following recursion:
\begin{align}
&K_0=w,\notag\\
&\d_x K_n=\frac{2u^2}{2n+1}\left(w\d_x+\frac{1}{2}w_x+\frac{1}{8}\d_x^3\right)K_{n-1},\quad\text{for $n\ge 1$}.\label{eq: kdv recursion}
\end{align}
It is a non-trivial fact that the right-hand side of \eqref{eq: kdv recursion} lies in the image of the operator $\d_x$ (see e.g. \cite{MJD00}). So the recursion~\eqref{eq: kdv recursion} determines a differential polynomial $K_n$ up to a polynomial in $u,u^{-1}$. This ambiguity should be fixed by the condition $\left.K_n\right|_{w_i=0}=0$.

Consider now the variable $w$ as a power series in variables $x,t_1,t_2,\ldots$ with the coefficients from $\mbC[u,u^{-1}]$. The KdV hierarchy is the following system of partial differential equations:
\begin{gather}\label{eq: kdv hierarchy}
\frac{\d w}{\d t_n}=\d_x K_n,\quad n\ge 1.
\end{gather}

Another form of Witten's conjecture says that the second derivative $\frac{\d^2 F^c}{\d t_0^2}$ is a solution of the KdV hierarchy~\eqref{eq: kdv hierarchy}. Here we identify $x$ with $t_0$. This form of Witten's conjecture is equivalent to the form that was stated in Section~\ref{subsection:Witten conjecture} (\cite{Wit91}). Let us formulate it precisely. Let $w(x=t_0,t_1,t_2,\ldots;u)$ be the solution of the KdV hierarchy~\eqref{eq: kdv hierarchy} specified by the initial condition $w|_{t_{\ge 1}=0}=u^{-2}x$. Let 
$$
S:=\frac{\d}{\d t_0}-\sum_{n\ge 0}t_{n+1}\frac{\d}{\d t_n}.
$$  
It is easy to show that $Sw=u^{-2}$ and there exists a unique power series $F(t_0,t_1,\ldots;u)$ such that $w=\frac{\d^2 F}{\d t_0^2}$, $S F=\frac{t_0^2}{2u^2}$ and $F|_{t_*=0}=0$. Moreover, we have (\cite{Wit91})
$$
\frac{\d^2 F}{\d t_0\d t_n}=\left.K_n\right|_{w_i=\frac{\d^{i+2}F}{\d t_0^{i+2}}}
$$ 
and, therefore, $F$ satisfies system~\eqref{eq:tau function}.


\section{Burgers-KdV hierarchy}\label{section:Burgers-KdV}

In this section we construct the Burgers-KdV hierarchy and prove that its flows satisfy the commutation relation from Lemma~\ref{lemma:existence}. This guarantees that the hierarchy has a solution for arbitrary polynomial initial conditions. We also construct a specific solution of the half of the Burgers-KdV hierarchy that satisfies the open string equation~\eqref{eq: open string}. Finally, we discuss a relation of the Burgers-KdV hierarchy to the equations for the wave function of the KdV hierarchy. 


\subsection{Construction}

Consider an extra variable~$v$ and the ring $\mcA_{v,w}$. Define differential polynomials $R_n,Q_n\in\mcA_{v,w}$, $n\ge 0$, as follows:
\begin{align*}
&R_0=v_x,\\
&R_n=\frac{2 u^2}{2n+1}\left[\left(\frac{1}{2}\d_x^2+v_x\d_x+\frac{v_x^2+v_{xx}}{2}+w\right)R_{n-1}+\frac{1}{2}v_x K_{n-1}+\frac{3}{4}\d_x K_{n-1}\right],\quad\text{for $n\ge 1$};\\
&Q_0=u\left(\frac{v_x^2+v_{xx}}{2}+w\right),\\
&Q_n=\frac{u^2}{n+1}\left(\frac{1}{2}\d_x^2+v_x\d_x+\frac{v_x^2+v_{xx}}{2}+w\right)Q_{n-1},\quad\text{for $n\ge 1$}.
\end{align*}

We call the Burgers-KdV hierarchy the following system:
\begin{align*}
&\frac{\d v}{\d t_n}=R_n,\quad n\ge 1; & & \frac{\d w}{\d t_n}=\d_xK_n,\quad n\ge 1;\\
&\frac{\d v}{\d s_n}=Q_n,\quad n\ge 0; & & \frac{\d w}{\d s_n}=0,\quad n\ge 0.
\end{align*}

We see that $w$ is just a solution of the KdV hierarchy. The simplest equation for $v$ is 
$$
\frac{\d v}{\d s_0}=u\left(\frac{v_x^2+v_{xx}}{2}+w\right).
$$
If we put $w=0$, then it coincides with the potential Burgers equation (see e.g. \cite{Olv86}). In this case the whole hierarchy reduces to the Burgers hierarchy (see e.g. \cite{Olv86}). This explains why we call the constructed hierarchy the Burgers-KdV hierarchy.

The system
\begin{align*}
&\frac{\d v}{\d t_n}=R_n,\quad n\ge 1;       & & \frac{\d w}{\d t_n}=\d_xK_n,\quad n\ge 1;\\
&\frac{\d v}{\d s}=u\left(\frac{v_x^2+v_{xx}}{2}+w\right); & & \frac{\d w}{\d s}=0.
\end{align*}
will be called the half of the Burgers-KdV hierarchy.

Another result of the paper is the following theorem.

\begin{theorem}\label{theorem: burgers}
Let $F^o$ be the power series determined by Theorem~\ref{theorem: existence}. Then the pair $v=F^o, w=\frac{\d^2 F^c}{\d t_0^2}$ satisfies the half of the Burgers-KdV hierarchy.
\end{theorem}
Here we again identify $x$ with $t_0$.


\subsection{Commutativity of the flows}

We are going to prove the following proposition.

\begin{proposition}\label{proposition: commutativity}
All operators $V_{R_i,\d_x K_i}$ and $V_{Q_i,0}$ commute with each other.
\end{proposition}

The proof of the proposition will occupy Sections~\ref{subsection: Burgers}-\ref{subsection: all flows}. The plan is the following. First, we consider the differential polynomials~$\tR_i:=\left.R_i\right|_{w_j=0},\tQ_i:=\left.Q_i\right|_{w_j=0}\in\mcA_{v}$ and show that the operators~$V_{\tR_i},V_{\tQ_i}\in\mcE_v$ pairwise commute. We do it in Section~\ref{subsection: Burgers}. Then in Section~\ref{subsection: commutativity with the first} we prove that the operators $V_{R_i,\d_x K_i}$ and $V_{Q_i,0}$ commute with $V_{Q_0,0}$. Finally, in Section~\ref{subsection: all flows} we deduce that all operators $V_{R_i,\d_x K_i}$ and $V_{Q_i,0}$ commute with each other.


\subsection{Burgers hierarchy}\label{subsection: Burgers}

Define an operator $B$ by $B:=\d_x+v_x$. It is easy to see that
$$
\frac{1}{2}B^2=\frac{1}{2}\d_x^2+v_x\d_x+\frac{v_x^2+v_{xx}}{2}.
$$
Therefore, we have
\begin{gather*}
\tR_i=\frac{u^{2i}}{(2i+1)!!}B^{2i}v_x,\qquad \tQ_i=\frac{u^{2i+1}}{2^i(i+1)!}B^{2i}\frac{v_x^2+v_{xx}}{2}.
\end{gather*}
We can easily recognize here the differential polynomials that describe the flows of the Burgers hierarchy (see~e.g.~\cite{Olv86}), up to multiplication by a constant. The fact that the operators~$V_{\tR_i},V_{\tQ_i}\in\mcE_v$ commute with each other is well-known (see~e.g.~\cite{Olv86}).


\subsection{Commutators $[V_{Q_0,0},V_{R_n,\d_x K_n}]$ and $[V_{Q_0,0},V_{Q_n,0}]$}\label{subsection: commutativity with the first}

Let
\begin{gather*}
P:=\frac{v_x^2+v_{xx}}{2},\qquad P_*:=\sum\frac{\d P}{\d v_i}\d_x^i=\frac{1}{2}\d_x^2+v_x\d_x.
\end{gather*}
The following formulas will be very useful for us:
\begin{gather*}
V_{Q_0,0}\cdot B=u[P_*,B]+uw_x,\qquad \frac{1}{2}B^2=P_*+P.
\end{gather*}

Let us prove that $[V_{Q_0,0},V_{Q_n,0}]=0$ or, equivalently, $V_{Q_0,0}Q_n-V_{Q_n,0}Q_0=0$. We have
\begin{align*}
&(n+1)u^{-2}\left(V_{Q_0,0}Q_n-V_{Q_n,0}Q_0\right)=V_{Q_0,0}\left(\left(\frac{1}{2}B^2+w\right)Q_{n-1}\right)-u P_*\left(\left(\frac{1}{2}B^2+w\right)Q_{n-1}\right)=\\
=&\frac{u}{2}[P_*,B^2]Q_{n-1}+\frac{u}{2}(w_xB+B\circ w_x)Q_{n-1}+\left(\frac{1}{2}B^2+w\right)\left(V_{Q_0,0}Q_{n-1}\right)\\
&-\frac{u}{2}(P_*\circ B^2)Q_{n-1}-u\left(P_*w\right)Q_{n-1}-uw_x\d_xQ_{n-1}-uwP_*Q_{n-1}=\\
=&\left(\frac{1}{2}B^2+w\right)\left(V_{Q_0,0}Q_{n-1}-V_{Q_{n-1},0}Q_0\right).
\end{align*}
Continuing in the same way we get 
$$
V_{Q_0,0}Q_n-V_{Q_n,0}Q_0=\frac{u^{2n}}{(n+1)!}\left(\frac{1}{2}B^2+w\right)^{n}\left(V_{Q_0,0}Q_{0}-V_{Q_{0},0}Q_0\right)=0.
$$

Let us prove that $[V_{Q_0,0},V_{R_n,\d_x K_n}]=0$. Since $K_n$ doesn't depend on $v_i$, we have $V_{Q_0,0}\d_x K_n=0$. Thus, it remains to prove that $V_{Q_0,0}R_n-V_{R_n,\d_x K_n}Q_0=0$. We proceed by induction on $n$. So assume that $V_{Q_0,0}R_{n-1}-V_{R_{n-1},\d_x K_{n-1}}Q_0=0$. 

We have
\begin{align*}
u^{-3}\frac{2n+1}{2}V_{Q_0,0}R_n=&u^{-1}V_{Q_0,0}\left[\left(\frac{1}{2}B^2+w\right)R_{n-1}+\frac{1}{2}v_x K_{n-1}+\frac{3}{4}\d_x K_{n-1}\right]\\
=&\frac{1}{2}\left([P_*,B^2]+w_xB+B\circ w_x\right)R_{n-1}+u^{-1}\left(\frac{1}{2}B^2+w\right)\left(V_{Q_0,0}R_{n-1}\right)\\
&+\frac{1}{2}(\d_x P+w_x)K_{n-1}=\\
=&\frac{1}{2}[P_*,B^2]R_{n-1}+w_xBR_{n-1}+\frac{1}{2}w_{xx}R_{n-1}\\
&+\left(\frac{1}{2}B^2+w\right)(P_*R_{n-1}+\d_x K_{n-1})+\frac{1}{2}(\d_x P+w_x)K_{n-1}=\\
=&P_*\left(\left(\frac{1}{2}B^2+w\right)R_{n-1}\right)+\frac{1}{2}B^2\left(\d_x K_{n-1}\right)+w\d_x K_{n-1}+\frac{1}{2}(\d_x P+w_x)K_{n-1}.
\end{align*}
On the other hand, we have
$$
u^{-3}\frac{2n+1}{2}V_{R_n,\d_xK_n}Q_0=P_*\left[\left(\frac{1}{2}B^2+w\right)R_{n-1}+\frac{1}{2}v_xK_{n-1}+\frac{3}{4}\d_xK_{n-1}\right]+\frac{2n+1}{2}u^{-2}\d_x K_n.
$$
Therefore, we get
\begin{multline*}
u^{-3}\frac{2n+1}{2}V_{Q_0,0}R_n-u^{-3}\frac{2n+1}{2}V_{R_n,\d_xK_n}Q_0=\\
=w\d_xK_{n-1}+\frac{1}{2}w_xK_{n-1}+\frac{1}{8}\d_x^3 K_{n-1}-\frac{2n+1}{2}u^{-2}\d_x K_n=0.
\end{multline*}
This completes our proof.


\subsection{All commutators}\label{subsection: all flows}

Let $V_1$ and $V_2$ be any two operators from the set $\{V_{R_i,\d_xK_i}\}_{i\ge 1}\cup\{V_{Q_i,0}\}_{i\ge 0}$. We have to prove that $[V_1,V_2]=0$. We begin with the following lemma.
\begin{lemma}
Let $[V_1,V_2]=V_{T_1,T_2}$. Then $T_2=0$ and $\left.T_1\right|_{w_*=0}=0$.
\end{lemma}
\begin{proof}
Let us prove that $T_2=0$. Clearly, we have to do it, only if $V_1,V_2\in \{V_{R_i,\d_xK_i}\}_{i\ge 1}$. Let $V_1=V_{R_i,\d_xK_i}$ and $V_2=V_{R_j,\d_x K_j}$. Since the differential polynomials~$K_l$ don't depend on $v_*$, we have 
$$
T_2=V_{R_i,\d_x K_i}\d_x K_j-V_{R_j,\d_x K_j}\d_x K_i=V_{\d_x K_i}\d_x K_j-V_{\d_x K_j}\d_x K_i.
$$
Here we consider the operators $V_{\d_x K_i}$ and $V_{\d_x K_j}$ as elements of the space $\mcE_w$. The last expression is equal to $0$, because the differential polynomials~$\d_x K_l$ describe the flows of the KdV hierarchy~(see~e.g.~\cite{Olv86}).

Let us prove that $\left.T_1\right|_{w_*=0}=0$. Let $\widetilde R_i:=\left.R_i\right|_{w_*=0}$ and $\widetilde Q_i:=\left.Q_i\right|_{w_*=0}$. We consider the operators $V_{\widetilde R_i}$ and $V_{\widetilde Q_i}$ as elements of $\mcE_v$. Obvioulsy, we have 
$$
\left(\left.V_{Q_i,0}R_j\right)\right|_{w_*=0}=V_{\widetilde Q_i}\widetilde R_j\text{ and }\left(\left.V_{Q_i,0}Q_j\right)\right|_{w_*=0}=V_{\widetilde Q_i}\widetilde Q_j.
$$
Since $\left.\d_x K_i\right|_{w_*=0}=0$, we get
$$
\left(\left.V_{R_i,\d_x K_i}R_j\right)\right|_{w_*=0}=V_{\widetilde R_i}\widetilde R_j\text{ and }\left(\left.V_{R_i,\d_x K_i}Q_j\right)\right|_{w_*=0}=V_{\widetilde R_i}\widetilde Q_j.
$$
In Section~\ref{subsection: Burgers} we showed that the operators $V_{\tR_i}$ and $V_{\tQ_j}$ pairwise commute. Thus, $\left.T_1\right|_{w_*=0}=0$. The lemma is proved.
\end{proof}

From the Jacobi identity it follows that $[V_{Q_0,0},[V_1,V_2]]=0$. We have proved that $[V_1,V_2]=V_{T_1,0}$, where $\left.T_1\right|_{w_*=0}$. The following lemma obviously completes the proof of Proposition~\ref{proposition: commutativity}.

\begin{lemma}
Suppose $T\in\mcA_{v,w}$ is an arbitrary differential polynomial such that $\left.T\right|_{w_*=0}=0$ and $[V_{Q_0,0},V_{T,0}]=0$. Then $T=0$.
\begin{proof}
Before proving the lemma let us introduce several notations. A partition $\lambda$ is a sequence of non-negative integers $\lambda_1,\ldots,\lambda_r$ such that $\lambda_1\ge\lambda_2\ge\ldots\ge\lambda_r$. Note that our terminology is slightly non-standard, because we allow zeroes in $\lambda$. Let $l(\lambda):=r$ and $|\lambda|:=\sum_{i=1}^r\lambda_i$. The set of all partitions will be denoted by $\mcP$. For a partition $\lambda$ let $v_\lambda:=\prod_{i=1}^{l(\lambda)}v_{\lambda_i}$ and $w_\lambda:=\prod_{i=1}^{l(\lambda)}w_{\lambda_i}$. 

Consider any differential polynomial $Q\in\mcA_{v,w}$. Let $Q=\sum_{\lambda,\mu\in\mcP} d_{\lambda,\mu}v_\lambda w_\mu$, where $d_{\lambda,\mu}\in\mbC[u,u^{-1}]$. Let
\begin{align*}
&Gr_i Q:=\sum_{l(\mu)=i}d_{\lambda,\mu}v_\lambda w_\mu,\qquad Gr_i^j Q:=\sum_{\substack{l(\mu)=i\\|\mu|=j}}d_{\lambda,\mu}v_\lambda w_\mu.
\end{align*}

The equation $[V_{Q_0,0},V_{T,0}]=0$ means that $V_{Q_0,0}T-V_{T,0}Q_0=0$. Let $T=\sum_{\substack{\lambda,\mu\in\mcP\\l(\mu)\ge 1}}c_{\lambda,\mu}v_\lambda w_\mu$, where $c_{\lambda,\mu}\in\mbC[u,u^{-1}]$. We have
\begin{align}
&V_{T,0}Q_0=u\left(v_x\d_x T+\frac{1}{2}\d_x^2 T\right),\label{eq: tmp1}\\
&V_{Q_0,0} T=u\sum_{\substack{\lambda,\mu\in\mcP\\l(\mu)\ge 1}}\sum_{i\ge 0}c_{\lambda,\mu}\frac{\d v_\lambda}{\d v_i}(\d_x^i P+w_i)w_\mu.\label{eq: tmp2}
\end{align}

Suppose $T\ne 0$. Let $i_0$ be the minimal $i$ such that $Gr_i T\ne 0$. From the condition $T|_{w_*=0}=0$ it follows that $i_0\ge 1$. Let $j_0$ be the maximal $j$ such that $Gr_{i_0}^j T\ne 0$. From \eqref{eq: tmp1} it is easy to see that
$$
Gr^{j_0+2}_{i_0}\left(V_{T,0}Q_0\right)=\frac{u}{2}\sum_{\substack{\lambda,\mu\in\mcP\\l(\mu)=i_0,|\mu|=j_0}}c_{\lambda,\mu}v_\lambda\d_x^2(w_\mu)\ne 0.
$$
On the other hand, from \eqref{eq: tmp2} it obviously follows that $Gr_{i_0}^{j_0+2}(V_{Q_0,0}T)=0$. This contradiction proves the lemma.

\end{proof}

\end{lemma}


\subsection{String solution}

In this section we construct a specific solution of the half of the Burgers-KdV hierarchy that satisfies the open string equation.

\begin{proposition}\label{proposition: string solution}
Consider the half of the Burgers-KdV hierarchy. Let us specify the following initial data for the hierarchy:
\begin{gather}\label{eq: initial condition for half}
\left.v\right|_{t_{\ge 1}=0,s=0}=0 \text{ and } \left.w\right|_{t_{\ge 1}=0,s=0}=u^{-2}x.
\end{gather}
Then the solution of the hierarchy satisfies the open string equation
$$
\frac{\d v}{\d t_0}=\sum_{n\ge 0}t_{n+1}\frac{\d v}{\d t_n}+u^{-1}s.
$$
\end{proposition}
We remind the reader that we identify $x$ with $t_0$.
\begin{proof}
Recall that $S:=\frac{\d}{\d t_0}-\sum_{n\ge 0}t_{n+1}\frac{\d}{\d t_n}$. We have to prove that $Sv=u^{-1}s$. We have $Sw=u^{-2}$, $w=\frac{\d^2 F}{\d t_0^2}$ and $K_n=\frac{\d^2 F}{\d t_0\d t_n}$ (see Section~\ref{subsection:kdv hierarchy}), therefore, $S K_n=K_{n-1}$, for $n\ge 1$. Let us put $K_{-1}:=u^{-2}$, so the last equation is also valid for $n=0$.

It is easy to see that $v|_{t_{\ge 1}=0}=u^{-1}\frac{s^3}{6}+u^{-1}t_0s$. Hence, 
\begin{gather}\label{eq: s-initial condition}
(Sv)|_{t_{\ge 1}=0}=u^{-1}s.
\end{gather}
For $n\ge 1$, we have
\begin{align*}
\frac{\d}{\d t_n}(Sv)=&S\frac{\d v}{\d t_n}-\frac{\d v}{\d t_{n-1}}=\\
=&\frac{2u^2}{2n+1}\left[\left((Sv)_x\d_x+(Sv)_x v_x+\frac{1}{2}(Sv)_{xx}+u^{-2}\right)\frac{\d v}{\d t_{n-1}}+\left(\frac{1}{2}B^2+w\right)S\frac{\d v}{\d t_{n-1}}+\right.\\
&+\left.\frac{1}{2}(Sv)_xK_{n-1}+\frac{1}{2}v_xK_{n-2}+\frac{3}{4}\d_x K_{n-2}\right]-\frac{\d v}{\d t_{n-1}}=\\
=&\frac{2u^2}{2n+1}\left[\left((Sv)_x\d_x+(Sv)_x v_x+\frac{1}{2}(Sv)_{xx}\right)\frac{\d v}{\d t_{n-1}}+\right.\\
&\left.+\left(\frac{1}{2}B^2+w\right)\frac{\d }{\d t_{n-1}}(Sv)+\frac{1}{2}(Sv)_xK_{n-1}\right].
\end{align*}
This system together with the initial condition \eqref{eq: s-initial condition} uniquely determines the power series $Sv$. It is easy to see that $Sv=u^{-1}s$ satisfies the system. Proposition~\ref{proposition: string solution} is proved.
\end{proof}

\subsection{Relation to the wave function of the KdV hierarchy}\label{subsection:wave}

Consider the operator
$$
L:=\d_x^2+2w.
$$
Recall that the KdV hierarchy can be written in the so-called Lax form:
$$
\frac{\d}{\d t_n}L=\frac{u^{2n}}{(2n+1)!!}\left[(L^{n+\frac{1}{2}})_+,L\right],\quad n\ge 1.
$$ 
Here we use the language of pseudo-differential operators. We briefly review it in~\cite{Bur14} and we refer the reader to the book~\cite{Dic03} for a detailed introduction to this subject. Introduce variables $t_n$ with $n\in\frac{1}{2}+\Z_{\ge 0}$ and let $t_{\frac{1}{2}+k}=s_k$, $k\in\Z_{\ge 0}$. In \cite{Bur14} we prove that the Burgers-KdV hierarchy is equivalent to the following system of evolutionary PDEs for the functions $w$ and $\phi=e^v$:
\begin{align}
&\frac{\d}{\d t_n}L=\frac{u^{2n}}{(2n+1)!!}\left[(L^{n+\frac{1}{2}})_+,L\right],\quad n\in\frac{1}{2}\Z_{\ge 1};\label{eq:Lax KdV}\\
&\frac{\d}{\d t_n}\phi=\frac{u^{2n}}{(2n+1)!!}(L^{n+\frac{1}{2}})_+\phi,\quad n\in\frac{1}{2}\Z_{\ge 1}.\label{eq:equations for wave}
\end{align}
Equations~\eqref{eq:equations for wave} coincide with the equations for the wave function of the KdV hierarchy~(see e.g.~\cite{Dic03}). The commutativity of the flows of the system~\eqref{eq:Lax KdV}-\eqref{eq:equations for wave} is actually well-known (see e.g.~\cite{Dic03}) and the proof is simple. Let us recall it. Consider the ring~$\mcA_{\phi,w}$. Let $T_n:=\frac{u^{2n}}{(2n+1)!!}(L^{n+\frac{1}{2}})_+\phi\in\mcA_{\phi,w}$. We have to check that the operators 
$$
V_{T_n,\d_x K_n}\in\mcE_{\phi,w},\quad n\in\frac{1}{2}\Z_{\ge 1},
$$ 
pairwise commute. Here we, by definition, put $K_n=0$, for $n\in\frac{1}{2}+\Z_{\ge 0}$. Let $m,n\in\frac{1}{2}\Z_{\ge 1}$ and
$$
[V_{T_m,\d_x K_m},V_{T_n,\d_x K_n}]=V_{P_1,P_2}.
$$
The fact that $P_2=0$ follows from the commutativity of the flows of the KdV hierarchy. For $P_1$ we have
\begin{align*}
P_1=&V_{T_m,\d_xK_m}T_n-V_{T_n,\d_xK_n}T_m=\\
=&\frac{u^{2m+2n}}{(2m+1)!!(2n+1)!!}\times\\
&\times\left(\left[(L^{m+\frac{1}{2}})_+,L^{n+\frac{1}{2}}\right]_++(L^{n+\frac{1}{2}})_+(L^{m+\frac{1}{2}})_+-\left[(L^{n+\frac{1}{2}})_+,L^{m+\frac{1}{2}}\right]_+-(L^{m+\frac{1}{2}})_+(L^{n+\frac{1}{2}})_+\right)\phi=\\
=&\frac{u^{2m+2n}}{(2m+1)!!(2n+1)!!}\left(\left[(L^{m+\frac{1}{2}})_+,(L^{n+\frac{1}{2}})_-\right]_+-\left[(L^{n+\frac{1}{2}})_+,L^{m+\frac{1}{2}}\right]_+\right)\phi.
\end{align*}
Since
$$
\left[(L^{n+\frac{1}{2}})_+,L^{m+\frac{1}{2}}\right]_+=-\left[(L^{n+\frac{1}{2}})_-,L^{m+\frac{1}{2}}\right]_+=-\left[(L^{n+\frac{1}{2}})_-,(L^{m+\frac{1}{2}})_+\right]_+,
$$
we conclude that $P_1=0$. The commutativity of the flows of the Burgers-KdV hierarchy is proved.


\section{Proofs of Theorems \ref{theorem: existence} and \ref{theorem: burgers}}\label{section: existence}

The proof of Theorems~\ref{theorem: existence} and \ref{theorem: burgers} goes as follows. Let a pair $(v,w)$ be the solution of the half of the Burgers-KdV hierarchy with the initial condition~\eqref{eq: initial condition for half}. Let us show that the series $F^o=v$ is a solution of the open KdV equations~\eqref{eq: open kdv}. Indeed, we have
\begin{align*}
\frac{2n+1}{2}\frac{\d v}{\d t_n}=&u^2\left[\left(\frac{1}{2}\d_x^2+v_x\d_x+\frac{v_x^2+v_{xx}}{2}+w\right)\frac{\d v}{\d t_{n-1}}+\frac{1}{2}v_x K_{n-1}+\frac{3}{4}\d_x K_{n-1}\right]=\\
=&u^2\left(\frac{v_x^2+v_{xx}}{2}+w\right)\frac{\d v}{\d t_{n-1}}+u^2\frac{\d}{\d t_{n-1}}\left(\frac{v_x^2+v_{xx}}{2}+w\right)+\frac{u^2}{2}v_xK_{n-1}-\frac{u^2}{4}\d_x K_{n-1}=\\
=&u\frac{\d v}{\d s}\frac{\d v}{\d t_{n-1}}+u\frac{\d^2 v}{\d s\d t_{n-1}}+\frac{u^2}{2}v_xK_{n-1}-\frac{u^2}{4}\d_x K_{n-1}.
\end{align*}
It remains to note that $K_n=\frac{\d^2 F^c}{\d t_0\d t_n}$ and we see that Theorems~\ref{theorem: existence} and~\ref{theorem: burgers} are proved.


\section{Proof of Theorem~\ref{theorem: main}}\label{section: open virasoro}

Denote by $a_{i,j}$ the number $\frac{(2i+1)!!(2j+1)!!}{2^{i+j+2}}$. Let $\tau:=\exp(F^o+F^c)$ and $\tau^o:=\exp(F^o)$. In order to save some space we will use the subscript~$n$ for the partial derivative by $t_n$ and the subscript~$s$ for the partial derivative by $s$. The proof of the theorem is based on the following lemma. 

\begin{lemma}\label{lemma: main lemma}
We have\\
1. $\frac{\cL_0\tau}{\tau}-(uF^o_s+u\d_s)\frac{\cL_{-1}\tau}{\tau}=0$.\\
2. If $n\ge 0$, then
\begin{align}
\label{eq: recursion}
\frac{\cL_{n+1}\tau}{\tau}-(uF^o_s+u\d_s)\frac{\cL_n\tau}{\tau}=&\frac{u^2}{4}\frac{(2n+1)!!}{2^n}F^c_{0,n}+\frac{u^4}{4}\sum_0^{n-1}a_{i,n-1-i}F_{0,i}^cF^c_{0,n-1-i}\\
&+\frac{u^2}{2}\frac{(2n+1)!!}{2^{n+2}}F^o_nF^o_0+\frac{u^2}{2}\sum_0^{n-1}a_{i,n-1-i}F_i^o\left(\frac{u^2}{2}F^o_0\d_x-\frac{u^2}{4}\d_x^2\right)F^c_{n-1-i}\notag\\
&+\frac{u^2}{2}\frac{(2n+1)!!}{2^{n+2}}F^o_{0,n}+\frac{u^4}{4}\sum_0^{n-1}a_{i,n-1-i}F^o_{0,i}F^c_{0,n-1-i}\notag\\
&-\frac{u^{n+1}}{4}\frac{\d_s^{n+1}\tau^o}{\tau^o}.\notag
\end{align}
\end{lemma}
\begin{proof}
Let us prove point {\it 1}. From the usual Virasoro equations~\eqref{eq: Virasoro} it follows that
\begin{align*}
&\frac{\cL_{-1}\tau}{\tau}=-\frac{\d F^o}{\d t_0}+\sum_{n\ge 0}t_{n+1}\frac{\d F^o}{\d t_n}+u^{-1}s,\\
&\frac{\cL_0\tau}{\tau}=-\frac{3}{2}\frac{\d F^o}{\d t_1}+\sum_{n\ge 0}\frac{2n+1}{2}t_n\frac{\d F^o}{\d t_n}+s\frac{\d F^o}{\d s}+\frac{3}{4}.
\end{align*}
Using the open KdV equations~\eqref{eq: open kdv} we get
\begin{multline*}
\frac{\cL_0\tau}{\tau}-(uF^o_s+u\d_s)\frac{\cL_{-1}\tau}{\tau}=\\
=-\frac{u^2}{2}F^o_0 F^c_{0,0}+\frac{u^2}{4}F^c_{0,0,0}+\frac{1}{2}t_0F^o_0+\sum_{n\ge 0}t_{n+1}\left(\frac{u^2}{2}F^o_0F^c_{0,n}-\frac{u^2}{4}F^c_{0,0,n}\right)+sF^o_s+\frac{3}{4}-(uF^o_s+u\d_s)u^{-1}s.
\end{multline*}
By the string equation~\eqref{eq: string equation}, the last expression is equal to zero.

Let us prove point {\it 2}. Let 
\begin{align*}
&L_n^{(1)}:=\sum_{i\ge 0}\frac{(2i+2n+1)!!}{2^{n+1}(2i-1)!!}(t_i-\delta_{i,1})\frac{\d}{\d t_{i+n}},\\
&L_n^{(2)}:=\frac{u^2}{2}\sum_{i=0}^{n-1}a_{i,n-1-i}\frac{\d^2}{\d t_i\d t_{n-1-i}},\\
&L_n^{(3)}:=u^n s\frac{\d^{n+1}}{\d s^{n+1}}+\frac{3n+3}{4}u^n\frac{\d^n}{\d s^n}.
\end{align*}
Using the Virasoro equations~\eqref{eq: Virasoro} we get
\begin{align}
&\frac{\cL_{n+1}\tau}{\tau}-(uF^o_s+u\d_s)\frac{\cL_n\tau}{\tau}=\notag\\
=&\frac{L^{(1)}_{n+1}\tau^o}{\tau^o}-(uF^o_s+u\d_s)\frac{L^{(1)}_n\tau^o}{\tau^o}\tag{A}\\
&+\frac{L^{(2)}_{n+1}\tau^o}{\tau^o}-(uF^o_s+u\d_s)\frac{L^{(2)}_n\tau^o}{\tau^o}\tag{B}\\
&+u^2\sum_0^n a_{i,n-i}F^c_i F^o_{n-i}-u^2(uF^o_s+u\d_s)\sum_0^{n-1} a_{i,n-1-i}F^c_i F^o_{n-1-i}\tag{C}\\
&+\frac{L^{(3)}_{n+1}\tau^o}{\tau^o}-(uF^o_s+u\d_s)\frac{L^{(3)}_n\tau^o}{\tau^o}\tag{D}.
\end{align}
Using the open KdV equations~\eqref{eq: open kdv} and the Virasoro equations~\eqref{eq: Virasoro} we can compute that
\begin{align*}
\text{(A)}=&\sum_0^\infty\frac{(2i+2n+1)!!}{2^{n+1}(2i-1)!!}(t_i-\delta_{i,1})\left(\frac{u^2}{2}F^o_0\d_x-\frac{u^2}{4}\d_x^2\right)F^c_{n+i}=\\
=&\underbrace{-\frac{u^2}{2}\frac{(2n+1)!!}{2^{n+1}}F^o_0F^c_n}_{*}-\frac{u^4}{4}\sum_0^{n-1}a_{i,n-1-i}F^o_0(\underbrace{F^c_{0,i,n-1-i}}_{**}+\underbrace{2F^c_{0,i}F^c_{n-1-i}}_{***})\\
&+\frac{u^2}{4}\frac{(2n+1)!!}{2^n}F^c_{0,n}+\frac{u^4}{8}\sum_0^{n-1}a_{i,n-1-i}(\underbrace{F^c_{0,0,i,n-1-i}}_{\bullet\bullet}+2F_{0,i}^cF^c_{0,n-1-i}\underbrace{+2F^c_{0,0,i}F^c_{n-1-i}}_{\bullet}).
\end{align*}
For expression (B) we have
\begin{align*}
\text{(B)}=&\frac{u^2}{2}\sum_0^n a_{i,n-i}(F^o_{i,n-i}+F^o_i F^o_{n-i})-\frac{u^2}{2}(uF^o_s+u\d_s)\sum_0^{n-1} a_{i,n-1-i}(F^o_{i,n-1-i}+F^o_i F^o_{n-1-i})=\\
=&\frac{u^2}{2}\frac{(2n+1)!!}{2^{n+2}}F^o_{0,n}+\frac{u^4}{4}\sum_0^{n-1}a_{i,n-1-i}F^o_{0,i}F^c_{0,n-1-i}+\frac{u^2}{2}\sum_0^{n-1}a_{i,n-1-i}\left(\underbrace{\frac{u^2}{2}F_0^o\d_x}_{**}\underbrace{-\frac{u^2}{4}\d_x^2}_{\bullet\bullet}\right)F^c_{i,n-1-i}\\
&+\frac{u^2}{2}\frac{(2n+1)!!}{2^{n+2}}F^o_nF^o_0+\frac{u^2}{2}\sum_0^{n-1}a_{i,n-1-i}F_i^o\left(\frac{u^2}{2}F^o_0\d_x-\frac{u^2}{4}\d_x^2\right)F^c_{n-1-i}.
\end{align*}
Computing (C) in a similar way we get
\begin{gather*}
\text{(C)}=\underbrace{u^2\frac{(2n+1)!!}{2^{n+2}}F^c_nF_0^o}_{*}+u^2\sum_0^{n-1}a_{i,n-1-i}F^c_i\left(\underbrace{\frac{u^2}{2}F_0^o\d_x}_{***}\underbrace{-\frac{u^2}{4}\d_x^2}_{\bullet}\right)F^c_{n-1-i}.
\end{gather*}
It is easy to compute that
\begin{gather*}
\text{(D)}=-\frac{u^{n+1}}{4}\frac{\d_s^{n+1}\tau^o}{\tau^o}.
\end{gather*}
We have marked the terms that cancel each other in the total sum (A)+(B)+(C)+(D). Collecting the remaining terms we get~\eqref{eq: recursion}.
\end{proof}

From the commutation relation~\eqref{eq: open commutation} it follows that $\cL_n=\frac{(-1)^{n-2}}{(n-2)!}ad_{\cL_1}^{n-2}\cL_2$, for $n\ge 3$. Thus, it is sufficient to prove the open Virasoro equations~\eqref{eq: open virasoro} only for $n=-1,0,1,2$. 

By Theorem~\ref{theorem: burgers} and Proposition~\ref{proposition: string solution}, the series $F^o$ satisfies the open string equation~\eqref{eq: open string}. Thus, $\cL_{-1}\tau=0$. By Lemma~\ref{lemma: main lemma}, $\cL_0\tau=\tau (uF^o_s+u\d_s)\frac{\cL_{-1}\tau}{\tau}=0$.

Substituting $n=0$ in \eqref{eq: recursion} we get
\begin{gather*}
\frac{\cL_1\tau}{\tau}-(uF^o_s+u\d_s)\frac{\cL_0\tau}{\tau}=\frac{u}{4}\left[u\left(F^c_{0,0}+\frac{1}{2}(F^o_0)^2+\frac{1}{2}F^o_{0,0}\right)-F^o_s\right]\stackrel{\text{by Theorem~\ref{theorem: existence}}}{=}0.
\end{gather*}
Therefore, $\cL_1\tau=0$.

Finally, let us write equation~\eqref{eq: recursion} for $n=1$. We get
\begin{align}
\label{eq: open1}
\frac{\cL_2\tau}{\tau}-(uF^o_s+u\d_s)\frac{\cL_1\tau}{\tau}=&\frac{3u^2}{8}F^c_{0,1}+\frac{u^4}{16}F_{0,0}^cF^c_{0,0}+\frac{3u^2}{16}F^o_1F^o_0+\frac{u^2}{8}F_0^o\left(\frac{u^2}{2}F^o_0\d_x-\frac{u^2}{4}\d_x^2\right)F^c_0\\
&+\frac{3u^2}{16}F^o_{0,1}+\frac{u^4}{16}F^o_{0,0}F^c_{0,0}-\frac{u^2}{4}\left((F^o_s)^2+F^o_{s,s}\right).\notag
\end{align}
Denote $\frac{\d^2 F^c}{\d t_0^2}$ by $w$ and $F^o$ by $v$. We have $F^c_{0,1}=K_1=\frac{w^2}{2}+\frac{1}{12}w_{xx}$. By the open KdV equations~\eqref{eq: open kdv}, we have
$$
F^o_1=u^2\left(\frac{v_x^3}{3}+v_xv_{xx}+\frac{v_{xxx}}{3}+v_xw+\frac{w_x}{2}\right).
$$
By Theorem~\ref{theorem: existence}, $F^o_s=u\left(\frac{v_x^2+v_{xx}}{2}+w\right)$. Substituting these expressions in the right-hand side of~\eqref{eq: open1}, after somewhat lengthy computations, we get zero. Hence, $\cL_2\tau=0$. The theorem is proved.


{\appendix

\section{Virasoro equations for the moduli of closed Riemann surfaces}\label{section: closed virasoro}

In this section we revisit the proof of the equivalence of the usual KdV and the Virasoro equations for the moduli space of stable curves. There is a reason to do it. In all papers, that we found, the proof is presented in a way more suitable for physicists. So we decided to rewrite it in a more mathematical style and also to make it more elementary. We follow the idea from~\cite{DVV91}.

Let $F$ be a power series in the variables $t_0,t_1,t_2,\ldots$ with the coefficients from $\mbC[u,u^{-1}]$. Suppose~$F$ satisfies the string equation~\eqref{eq: string equation}, the condition~$F|_{t_*=0}=0$ and the second derivative~$\frac{\d^2 F}{\d t_0^2}$ is a solution of the KdV hierarchy~\eqref{eq: kdv hierarchy}. Let us prove that $F$ satisfies the Virasoro equations~\eqref{eq: Virasoro}.

Let $D:=F_{0,0}\d_x+\frac{1}{2}F_{0,0,0}+\frac{1}{8}\d_x^3$ and $w=\frac{\d^2 F}{\d t_0^2}$. We can rewrite the KdV equations~\eqref{eq:tau function} in the following way:
\begin{gather}\label{eq: kdv}
u^2 D F_{0,n-1}=\frac{2n+1}{2}\d_xF_{0,n}.
\end{gather}
Let $\tau:=\exp\left(F\right)$.

\begin{lemma}\label{lemma: virasoro lemma}
For any $n\ge -1$, we have 
$$
u^2D\d_x\frac{L_n\tau}{\tau}=\d_x^2\frac{L_{n+1}\tau}{\tau}.
$$
\end{lemma}
\begin{proof}
Suppose $n=-1$. We have
\begin{align*}
&\frac{L_{-1}\tau}{\tau}=-F_0+\sum_{n\ge 0}t_{n+1}F_n+\frac{t_0^2}{2u^2},\\
&\frac{L_0\tau}{\tau}=-\frac{3}{2}F_1+\sum_{n\ge 0}\frac{2n+1}{2}t_n F_n+\frac{1}{16}.
\end{align*}
From the KdV equations~\eqref{eq: kdv} it follows that
\begin{gather*}
u^2D\d_x\frac{L_{-1}\tau}{\tau}-\d_x^2\frac{L_0\tau}{\tau}=\frac{1}{2}D\d_x(t_0^2)-\frac{1}{2}\d_x^2(t_0F_0)=0.
\end{gather*}

Suppose $n\ge 0$. The operators $L_n^{(1)}$ and $L_n^{(2)}$ were defined in Section~\ref{section: open virasoro}. Using the KdV equations~\eqref{eq: kdv} we get 
\begin{multline*}
u^2 D\d_x\frac{L_n^{(1)}\tau}{\tau}-\d_x^2\frac{L_{n+1}^{(1)}\tau}{\tau}=\\
=\frac{(2n+1)!!}{2^{n+1}}\left(2u^2 F_{0,0}F_{0,n}+\frac{u^2}{2}F_{0,0,0,n}-(2n+3)F_{0,n+1}\right)+\frac{(2n+1)!!}{2^{n+2}}u^2 F_{0,0,0}F_n.
\end{multline*}
Now let us compute $u^2 D\d_x\frac{L_n^{(2)}\tau}{\tau}-\d_x^2\frac{L_{n+1}^{(2)}\tau}{\tau}$. Recall that $a_{i,j}:=\frac{(2i+1)!!(2j+1)!!}{2^{i+j+2}}$. We have 
$$
\frac{L_n^{(2)}\tau}{\tau}=\frac{u^2}{2}\sum_{i=0}^{n-1}a_{i,n-i-1}(F_{i,n-1-i}+F_iF_{n-1-i}).
$$
Using the KdV equations~\eqref{eq: kdv} it is easy to compute that
\begin{multline*}
\frac{u^2}{2}\left[u^2\sum_{i=0}^{n-1}a_{i,n-i-1}D\d_x F_{i,n-1-i}-\sum_{i=0}^na_{i,n-i}\d_x^2 F_{i,n-i}\right]=\\
=-\frac{u^4}{2}\sum_{i=0}^{n-1}a_{i,n-i-1}(F_{0,0,i}F_{0,0,n-1-i}+\frac{1}{2}F_{0,0,0,i}F_{0,n-1-i})-\frac{u^2}{2}\frac{(2n+1)!!}{2^{n+2}}F_{0,0,0,n}.
\end{multline*}
By the KdV equations~\eqref{eq: kdv}, we have
\begin{align*}
&\frac{u^2}{2}\left[u^2D\d_x\sum_{i=0}^{n-1}a_{i,n-1-i}F_iF_{n-1-i}-\d_x^2\sum_{i=0}^n a_{i,n-i}F_iF_{n-i}\right]=\\
&\phantom{aaaaaaaaaaaaa}=\frac{u^4}{2}\sum_{i=0}^{n-1}a_{i,n-i-1}(2F_{0,0}F_{0,i}F_{0,n-1-i}+F_{0,i}F_{0,0,0,n-1-i}+\frac{3}{4}F_{0,0,i}F_{0,0,n-1-i})\\
&\phantom{aaaaaaaaaaaaa=}-u^2\sum_{i=0}^n a_{i,n-i}F_{0,i}F_{0,n-i}-u^2\frac{(2n+1)!!}{2^{n+2}}F_{0,0,0}F_n.
\end{align*}

Collecting all the terms we get
\begin{align*}
u^2 D\d_x\frac{L_n\tau}{\tau}-\d_x^2\frac{L_{n+1}\tau}{\tau}=&\frac{(2n+1)!!}{2^{n+1}}\left(2u^2 F_{0,0}F_{0,n}+\frac{u^2}{2}F_{0,0,0,n}-(2n+3)F_{0,n+1}\right)\\
&-\frac{u^4}{2}\sum_{i=0}^{n-1}a_{i,n-i-1}(F_{0,0,i}F_{0,0,n-1-i}+\frac{1}{2}F_{0,0,0,i}F_{0,n-1-i})-\frac{u^2}{2}\frac{(2n+1)!!}{2^{n+2}}F_{0,0,0,n}\\
&+\frac{u^4}{2}\sum_{i=0}^{n-1}a_{i,n-i-1}(2F_{0,0}F_{0,i}F_{0,n-1-i}+F_{0,i}F_{0,0,0,n-1-i}+\frac{3}{4}F_{0,0,i}F_{0,0,n-1-i})\\
&-u^2\sum_{i=0}^n a_{i,n-i}F_{0,i}F_{0,n-i}=\\
=&\frac{(2n+1)!!}{2^n}\left(u^2 F_{0,0}F_{0,n}+\frac{u^2}{8}F_{0,0,0,n}-\frac{2n+3}{2}F_{0,n+1}\right)\\
&+u^4\sum_{i=0}^{n-1}a_{i,n-i-1}\left(-\frac{1}{8}F_{0,0,i}F_{0,0,n-1-i}+F_{0,0}F_{0,i}F_{0,n-1-i}+\frac{1}{4}F_{0,i}F_{0,0,0,n-1-i}\right)\\
&-u^2\sum_{i=0}^na_{i,n-i}F_{0,i}F_{0,n-i}=:Q.
\end{align*}
Since $K_i=\frac{\d^2 F}{\d t_0\d t_i}$, the series $Q$ can be expressed as a differential polynomial in $w,w_x,w_{xx},\ldots$. From the condition $\left.K_i\right|_{w_*=0}=0$ it follows that this polynomial doesn't have a constant term. Using~\eqref{eq: kdv recursion} it is easy to compute that $\d_xQ=0$. Thus, $Q=0$. The lemma is proved.
\end{proof}

From the commutation relation~\eqref{eq: commutation} it follows that $L_n=\frac{(-1)^{n-2}}{(n-2)!}ad_{L_1}^{n-2}L_2$, for $n\ge 3$. Thus, it is sufficient to prove the Virasoro equations~\eqref{eq: virasoro} only for $n=-1,0,1,2$. Let us do it by induction on $n$. 

The case $n=-1$ follows from the string equation. Suppose $n\ge 0$. Recall that $S:=\frac{\d}{\d t_0}-\sum_{i\ge 0}t_{i+1}\frac{\d}{\d t_i}$. Using the induction hypothesis we get
\begin{gather*}
S\frac{L_n\tau}{\tau}=\frac{S(L_n\tau)}{\tau}-\frac{(S\tau)L_n\tau}{\tau^2}=-\frac{L_{-1}(L_n\tau)}{\tau}=(n+1)\frac{L_{n-1}\tau}{\tau}-\frac{L_n(L_{-1}\tau)}{\tau}=0.
\end{gather*}
Therefore, $S\d_x\frac{L_n\tau}{\tau}=0$. By Lemma~\ref{lemma: virasoro lemma} and the induction assumption, we have $\d_x^2\frac{L_n\tau}{\tau}=0$, hence, $\left(\sum_{i\ge 0}t_{i+1}\frac{\d}{\d t_i}\right)\d_x\frac{L_n\tau}{\tau}=0$. Therefore, $\d_x\frac{L_n\tau}{\tau}\in\mbC[u,u^{-1}]$. Since $S\frac{L_n\tau}{\tau}=0$, we have $\left(\sum_{i\ge 0}t_{i+1}\frac{\d}{\d t_i}\right)\frac{L_n\tau}{\tau}\in\mbC[u,u^{-1}]$. From this we conclude that $\frac{L_n\tau}{\tau}\in\mbC[u,u^{-1}]$. Let $\<\tau_{a_1}\ldots\tau_{a_k}\>:=\left.\frac{\d^k F}{\d t_{a_1}\ldots\d t_{a_k}}\right|_{t_*=0}$. We have
\begin{align*}
&\left.\frac{L_0 \tau}{\tau}\right|_{t_*=0}=-\frac{3}{2}\<\tau_1\>+\frac{1}{16},\\
&\left.\frac{L_1 \tau}{\tau}\right|_{t_*=0}=-\frac{15}{4}\<\tau_2\>+\frac{u^2}{8}(\<\tau_0^2\>+\<\tau_0\>^2),\\
&\left.\frac{L_2 \tau}{\tau}\right|_{t_*=0}=-\frac{105}{8}\<\tau_3\>+\frac{3u^2}{8}(\<\tau_0\tau_1\>+\<\tau_0\>\<\tau_1\>).
\end{align*}
It is easy to check that all these expressions are equal to zero. The Virasoro equations are proved.
}

\end{document}